\documentclass[12pt,leqno]{amsart}
\usepackage{amsmath}
\usepackage{amssymb}
\def\overset#1#2{{\mathrel{\mathop {{#2}_{}}\limits^{#1}}}}
\def\underset#1#2{{\mathrel{\mathop {{}_{} {#2}}\limits_{{#1}_{}}}}}
\def\upplim_#1{\underset{#1}{\overline\lim}\;}
\def\lowlim_#1{\underset{#1}{\underline\lim}\;}
\newtheorem{definition}[equation]{Definition}

\newtheorem{lemma}[equation]{Lemma}

\newtheorem{theorem}[equation]{Theorem}

\newcommand{\C}{{\mathbb{C}}}
\newcommand{\N}{\mathbb{N}}
\renewcommand{\P}{{\mathbb{P}}}

\newcommand{\R}{{\mathbb{R}}}

\numberwithin{equation}{section}

\title[A generalization of the subspace theorem for higher degree polynomials]{A generalization of the subspace theorem for higher degree polynomials in subgeneral position} 

\author{Si Duc Quang}

\begin{document}

\maketitle 

\begin{abstract}
In this paper, we prove a generalization of the Schmidt's subspace theorem for polynomials of higher degree in subgeneral position with respect to a projective variety over a number field. Our result improves and generalizes the previous results on Schmidt's subspace theorem for the case of higher degree polynomials. 
\end{abstract}

\def\thefootnote{\empty}

\section{Introduction}

In Diophantine approximation, the classical theorem of Roth \cite{R} says that for an algebraic number $\alpha$ there are only finite rational numbers, whose denominators satisfy a certain bounded condition, close enough $\alpha$. This result is stated as follows.

\vskip0.2cm
\noindent
\textbf{Theorem A.} (Roth's theorem \cite{R}) {\it Let $\alpha\in\overline{Q}$ be an algebraic number. Let $\epsilon >0$. Then there are only finitely many rational number $p/q\in\mathbb Q$ satisfying
$$ \left |\alpha-\frac{p}{q}\right|<\frac{1}{q^{2+\epsilon}}.$$}

By Dirichlet's principle, the exponent $(2+\epsilon)$ is the best possible. Later on, this result had been extended to the case of  arbitrary number field instead of $\mathbb Q$ and also for a finite set of absolute values, not for only archimedean absolute values as above. We refer the readers to the works of Wirsing \cite{W} and Schmidt \cite{Sch70} for this direction. 

In another direction, Schmidt \cite[Lemma 7]{Sch72} proved an interesting generalization of the Roth's theorem to the case of higher dimension, which is so called ``subspace theorem''. There are many authors extending and generalizing this result of Schmidt.  Over the last decades much research on  the subspace theorem has been done. Many results have been given, for instance \cite{EF1,EF2,CZ,RW,L}. To state some of them, we recall the following.

Let $k$ be a number field. Denote by $M_k$ the set of places (equivalence classes of absolute values) of $k$ and by $M^\infty_k$ the set of archimedean places of $k$. For each $v\in M_k$, we choose the normalized absolute value $|\cdot |_v$ such that $|\cdot |_v=|\cdot|$ on $\mathbb Q$ (the standard absolute value) if $v$ is archimedean, and $|p|_v=p^{-1}$ if $v$ is non-archimedean and lies above the rational prime $p$. For each $v\in M_k$, denote by $k_v$ the completion of $k$ with respect to $v$ and set $$n_v := [k_v :\mathbb Q_v]/[k :\mathbb Q].$$
We put $||x||_v=|x|^{n_v}_v$. The product formula is stated as follows
$$\prod_{v\in M_k}||x||_v=1,\text{ for }  x\in k^*.$$

For $x = (x_0 ,\ldots , x_m)\in k^{m+1}$, define
$$||x||_v :=\max\{||x_0||_v,\ldots,||x_m||_v\},\ v\in M_k.$$
We define the absolute logarithmic height of a projective point $x=(x_0:\cdots :x_m)\in\P^m(k)$ by
$$h(x):=\sum_{v\in M_k}\log ||x||_v.$$
By the product formula, this does not depend on the choice of homogeneous coordinates $(x_0:\cdots :x_m)$ of $x$. If $x\in k^*$, we define the absolute logarithmic height of $x$ by
$$h(x):=\sum_{v\in M_k}\log^+ ||x||_v,$$
where $\log^+a=\log\max\{1,a\}.$

Let $\N=\{0,1,2,\ldots\}$ and for a positive integer $d$, we set
$$\mathcal T_d :=\{(i_0,\ldots, i_m)\in\N^{m+1}\ :\ i_0+\cdots +i_m=d\}.$$
Let $Q=\sum_{I\in\mathcal T_d}a_I{\bf x}^I$ be a homogeneous polynomial of degree $d$ in $k[x_0,\ldots, x_m]$, where ${\bf x}^I = x^{i_0}\ldots x^{i_m}_m$ for ${\bf x}=(x_0,\ldots, x_m)$  and $I = (i_0,\ldots,i_m)$. Define $||Q||_v =\max\{||a_I||_v; I\in\mathcal T_d\}$. The height of $Q$ is defined by
$$h(Q)=\sum_{v\in M_k}\log ||Q||_v.$$
For each $v\in M_k$, we define the Weil function $\lambda_{Q,v}$ by
$$\lambda_{Q,v}({\bf x}):=\log\frac{||{\bf x}||_v^d\cdot ||Q||_v}{||Q({\bf x})||_v},\ {\bf x}\in\P^m(k)\setminus\{Q=0\}.$$

\begin{definition}\label{def1.6}{\rm Let $V$ be an irreducible projective subvariety of $\P^m(k)$ of dimension $n$. We say that a set $\{Q_j\}^q_{j=1}\ (q\ge N+1)$ of homogeneous polynomials in $\overline{k}[x_0,\ldots, x_m]$ is in $N$-subgeneral position with respect to $V$ if for any $1\le j_0 <\cdots < j_{N+1}\le q$, 
$$ V(\overline{k})\cap\{Q_{j_0}=0\}\cap\ldots\cap\{Q_{j_N}=0\}=\varnothing.$$
where $\overline{k}$ is an algebraic closure of $k$. Moreover, if $\{Q_j\}^q_{j=1}\ (q\ge N+1)$ is in $n$-subgeneral position w.r.t. $V$ then it is said to be in general position w.r.t. $V$.}
\end{definition}

We state here a general subspace theorem given by Vojta for the case of hyperplanes of a projective space in general position (i.e., the case of linear forms in general position).

\vskip0.2cm
\noindent
\textbf{Theorem B} (see \cite[Theorem 2.2.4]{V87} and \cite[Theorem 0.3]{V89}). {\it Let $k$ be a number field, $S$ be a finite set of places of $k$ containing all archimedean places and let $H_1,\ldots, H_q$ be hyperplanes of $\P^n(k)$ in general position. Then for each $\epsilon >0$, 
$$\sum_{v\in S}\sum_{j=1}^q\lambda_{H_j,v}({\bf x})\le ( n+1+\epsilon)h({\bf x})$$
for all ${\bf x}\in\P^n(k)$ outside a union of finite proper linear subspaces.}
 
\vskip0.2cm
As we known that there is a close relation between Nevanlinna theory and Diophantine Approximation due to the works of Osgood (see \cite{Os1, Os2}) and Vojta (see \cite{V87}). Especially, Vojta has given a dictionary which provides the correspondences for the fundamental concepts of these two theories. With this dictionary, the subspace theorem in Diophantine Approximation theory will correspond to the second main theorem in Nevanlinna theory. In 1991, by using the notion of Nochka weights in Nevanlinna theory \cite{Noc83}, Ru and Wong \cite{RW} gave a generalization of Theorem B as follows.
 
\vskip0.2cm
\noindent
\textbf{Theorem C} (see \cite[Theorem 4.1]{RW}). {\it Let $k$ be a number field, $S$ be a finite set of places of $k$ containing all archimedean places and let $H_1,\ldots, H_q$ be hyperplanes in $\P^n(k)$ in $N-$subgeneral position. Then for each $\epsilon >0$, 
$$\sum_{v\in S}\sum_{j=1}^q\lambda_{H_j,v}({\bf x})\le (2N- n+1+\epsilon)h({\bf x})$$
for all ${\bf x}\in\P^n(k)$ outside a union of finite proper linear subspaces.}

Later on, Corvaja-Zannier \cite{CZ} and Evertse-Ferretti \cite{EF2} gave a breakthrough result which generalizes the result of Vojta to the case where the hyperplanes are replaced by hypersurfaces. We state their results in a simple form as follows.

\vskip0.2cm
\noindent
\textbf{Theorem D} (see \cite[Theorem 1.1]{EF2} and also \cite[Theorem 1.3]{CZ}). {\it Let $k$ be a number field, $S$ be a finite set of places of $k$ containing all archimedean places and let $V$ be an irreducible subvariety of $\P^m(k)$ of dimension $n$. Let $Q_1,\ldots, Q_q$ be homogeneous polynomials in $\overline{k}[x_0,...,x_m]$ in general position with respect to $V$. Then for each $\epsilon >0$, 
$$\sum_{v\in S}\sum_{j=1}^q\dfrac{\lambda_{Q_j,v}({\bf x})}{\deg Q_j}\le (n+1+\epsilon)h({\bf x})$$
for all ${\bf x}\in\P^n(k)$ outside a union of finite proper algebraic subvarieties.}

\vskip0.2cm
Here, we note that Corvaja and Zannier considered the above result for the case of $V=\P^n(k)$ and the general case of arbitrary subvariety $X$ is proved by Evertse and Ferretti. Also the methods of two these groups of authors are different. While Corvaja and Zannier introduced a method of filtrating the spaces of homogeneous polynomials,  Evertse and Ferretti gave a method of making use of Chow weights. 

Motivated by the analogy between Nevanlinna theory and Diophantine approximation, after establishing a degenerated second main theorem, Chen, Ru and Yan \cite{CRY} proved a Schmidt's subspace theorem for the case of hypersurface in subgeneral position as follows. 

\vskip0.2cm
\noindent 
\textbf{Theorem E} (see \cite[Theorem 1.3]{CRY})\ {\it Let $k$ be a number field, $S$ be a finite set of places of $k$ and let $V$ be an irreducible projective subvariety of $\P^m(k)$ of dimension $n$. Let $Q_1,\ldots, Q_q$ be homogeneous polynomials of $\overline{k}[x_0,\ldots,x_m]$ in general position with respect to $V$. Then for each $\epsilon >0$, 
$$\sum_{v\in S}\sum_{j=1}^q\dfrac{\lambda_{Q_j,v}({\bf x})}{\deg Q_j}\le (N(n+1)+\epsilon)h({\bf x})$$
for all ${\bf x}\in\P^m(k)$ outside a union of closed proper subvarieties of $V$.}

\vskip0.2cm 
However, while Theorem C of Ru and Wong  is a natural generalization of the classical result of Vojta, the above result of Chen, Ru and Yan can not imply the result of Corvaja- Zannier- Evertse- Ferretti. Recently, in \cite{Q16-1}, by introducing the method of ``replacing hypersurfaces'', we gave a general second main theorem in Nevanlinna theory for hypersurfaces in subgeneral position with respect to a subvariety. Adopting our technique in \cite{Q16-1}, in this paper we will prove a Schmidt's subspace theorem for homogeneous polynomials in subgeneral position with respect to a projective variety. Our main result is stated as follows.

\vskip0.2cm
\noindent 
\textbf{Main Theorem.}\ {\it Let $k$ be a number field, $S$ be a finite set of places of $k$ and let $V$ be an irreducible projective subvariety of $\P^m(k)$ of dimension $n$. Let $Q_1,\ldots, Q_q$ be homogeneous polynomials of $\overline{k}[x_0,\ldots,x_m]$ in general position with respect to $V$. Then for each $\epsilon >0$, 
$$\sum_{v\in S}\sum_{j=1}^q\dfrac{\lambda_{Q_j,v}({\bf x})}{\deg Q_j}\le ((N-n+1)(n+1)+\epsilon)h({\bf x})$$
for all ${\bf x}\in\P^m(k)$ outside a union of closed proper subvarieties of $V$.}

\vskip0.2cm 
We would like to note that, when the family of homogeneous polynomials is in general position, i.e., $N=n$, our above result will imply the Schmidt's subspace theorems of  Corvaja-Zannier-Evertse-Ferretti and also previous authors.

\section{Some auxiliary results}

In this section, we recall the notion of Chow weights and Hilbert weights from \cite{EF1,EF2}.

Let $X\subset\P^m(k)$ be a projective variety of dimension $n$ and degree $\Delta$.  We consider the unique polynomial
$$F_X(\textbf{u}_0,\ldots,\textbf{u}_k) = F_X(u_{00},\ldots,u_{0m};\ldots; u_{n0},\ldots,u_{nm})$$
in $n+1$ blocks of variables $\textbf{u}_i=(u_{i0},\ldots,u_{im}), i = 0,\ldots,k$. Then $F_X$ is called the Chow form associated to $X$ and has the following
properties: 
\begin{itemize}
\item $F_X$ is irreducible in $\C[u_{00},\ldots,u_{nm}]$,
\item $F_X$ is homogeneous of degree $\Delta$ in each block $\textbf{u}_i, i=0,\ldots,n$,
\item $F_X(\textbf{u}_0,\ldots,\textbf{u}_n) = 0$ if and only if $X\cap H_{\textbf{u}_0}\cap H_{\textbf{u}_n}\ne\varnothing$, where $H_{\textbf{u}_i}, i = 0,\ldots,n$, are the hyperplanes given by $u_{i0}x_0+\cdots+ u_{im}x_m=0.$
\end{itemize}
Let ${\bf c}=(c_0,\ldots, c_m)$ be a tuple of real numbers. Let $t$ be an auxiliary variable. We consider the decomposition
\begin{align*}
F_X(t^{c_0}u_{00},&\ldots,t^{c_m}u_{0m};\ldots ; t^{c_0}u_{n0},\ldots,t^{c_m}u_{nm})\\ 
& = t^{e_0}G_0(\textbf{u}_0,\ldots,\textbf{u}_m)+\cdots +t^{e_r}G_r(\textbf{u}_0,\ldots, \textbf{u}_m).
\end{align*}
with $G_0,\ldots,G_r\in\C[u_{00},\ldots,u_{0m};\ldots; u_{n0},\ldots,u_{nm}]$ and $e_0>e_1>\cdots>e_r$. The Chow weight of $X$ with respect to ${\bf c}$ is defined by
\begin{align*}
e_X({\bf c}):=e_0.
\end{align*}
For each subset $J = \{j_0,...,j_k\}$ of $\{0,...,n\}$ with $j_0<j_1<\cdots<j_k,$ we define the bracket
\begin{align*}
[J] = [J]({\bf u}_0,\ldots,{\bf u}_n):= \det (u_{ij_t}), i,t=0,\ldots ,k,
\end{align*}
where $\textbf{u}_i = (u_{i_0},\ldots ,u_{im})$ denotes the blocks of $m+1$ variables. Let $J_1,\ldots ,J_\beta$ with $\beta=\binom{n+1}{m+1}$ be all subsets of $\{0,...,m\}$ of cardinality $n+1$.

Then the Chow form $F_X$ of $X$ can be written as a homogeneous polynomial of degree $\Delta$ in $[J_1],\ldots,[J_\beta]$. We may see that for $\textbf{c}=(c_0,\ldots,c_m)\in\R^{m+1}$ and for any $J$ among $J_1,\ldots,J_\beta$,
\begin{align*}
[J](t^{c_0}u_{00},\ldots,t^{c_m}u_{0m},&\ldots,t^{c_0}u_{n0},\ldots,t^{c_m}u_{nm})\\
&=t\sum_{j\in J}c_j[J](u_{00},\ldots,u_{0m},\ldots,u_{n0},\ldots,u_{nm}).
\end{align*}

For $\textbf{a} = (a_0,\ldots,a_m)\in\mathbb Z^{m+1}$ we write ${\bf x}^{\bf a}$ for the monomial $x^{a_0}_0\cdots x^{a_m}_m$. Let $I=I_X$ be the prime ideal in $\C[x_0,\ldots,x_m]$ defining $X$. Let $k[x_0,\ldots,x_m]_u$ denote the vector space of homogeneous polynomials in $k[x_0,\ldots,x_m]$ of degree $u$ (including $0$). Put $I_u :=k[x_0,\ldots,x_m]_u\cap I$ and define the Hilbert function $H_X$ of $X$ by, for $u = 1, 2,...,$
\begin{align*}
H_X(u):=\dim (k[x_0,...,x_n]_u/I_u).
\end{align*}
By the usual theory of Hilbert polynomials,
\begin{align*}
H_X(u)=\Delta\cdot\frac{u^m}{m!}+O(u^{m-1}).
\end{align*}
The $u$-th Hilbert weight $S_X(u,{\bf c})$ of $X$ with respect to the tuple ${\bf c}=(c_0,\ldots,c_m)\in\mathbb R^{m+1}$ is defined by
\begin{align*}
S_X(u,{\bf c}):=\max\left (\sum_{i=1}^{H_X(u)}{\bf a}_i\cdot{\bf c}\right),
\end{align*}
where the maximum is taken over all sets of monomials ${\bf x}^{{\bf a}_1},\ldots,{\bf x}^{{\bf a}_{H_X(u)}}$ whose residue classes modulo $I$ form a basis of $k[x_0,\ldots,x_n]_u/I_u.$

The following theorems are due to Evertse and Ferretti \cite{EF1,EF2}.
\begin{theorem}[{see \cite[Theorem 4.1]{EF1}}]\label{2.12}
Let $X\subset\P^m(k)$ be an algebraic variety of dimension $n$ and degree $\Delta$. Let $u>\Delta$ be an integer and let ${\bf c}=(c_0,\ldots,c_u)\in\mathbb R^{u+1}_{\geqslant 0}$.
Then
$$ \frac{1}{uH_X(u)}S_X(u,{\bf c})\ge\frac{1}{(m+1)\Delta}e_X({\bf c})-\frac{(2m+1)\Delta}{u}\cdot\left (\max_{i=0,...,u}c_i\right). $$
\end{theorem}

\begin{lemma}[{see \cite[Lemma 5.1]{EF2}}]\label{2.13}
Let $Y$ be a subvariety of $\P^{q-1}(k)$ of dimension $n$ and degree $\Delta$. Let ${\bf c}=(c_1,\ldots, c_q)$ be a tuple of positive reals. Let $\{i_0,...,i_n\}$ be a subset of $\{1,...,q\}$ such that
$$Y \cap \{y_{i_0}=\cdots =y_{i_n}=0\}=\varnothing.$$
Then
$$e_Y({\bf c})\ge (c_{i_0}+\cdots +c_{i_n})\Delta.$$
\end{lemma}

\section{Proof of Main Theorem}
The following lemma is firstly introduced in \cite{Q16-1} for the case $k=\C$ and is reproved in \cite{Q16-2} for the general case of arbitrary number field and for even \textit{moving polynomials}.

\begin{lemma}[{see \cite[Lemma 3.1]{Q16-1} and also \cite[Lemma 3.1]{Q16-2}}]\label{3.1}
Let $V$ be an irreducible projective subvariety of $\P^m(k)$ of dimension $n$. Let $Q_1,...,Q_{N+1}\ (N\ge n)$ be homogeneous polynomials in $\overline{k}[x_0,...,x_m]$ of the same degree $d\ge 1$, such that
$$V(\overline{k})\cap\bigcap_{i=1}^{N+1}\{Q_i=0\}=\varnothing.$$
Then there exists $n$ homogeneous polynomial $P_{2},...,P_{n+1}$ of the forms
$$P_t=\sum_{j=2}^{N-n+t}c_{tj}Q_j, \ c_{tj}\in\C,\ t=2,...,n+1,$$
such that $\left (\bigcap_{t=1}^{n+1}\{P_t=0\}\right )\cap V(\overline{k})=\varnothing,$ where $P_1=Q_1$.
\end{lemma}
\begin{proof}[{\bf Proof of Main Theorem}]
Let $d_i=\deg Q_i\ (1\le i\le q)$ and let $d$ be the least common multiple of $d_i$'s, i.e., $d=l.c.m(d_1,\ldots,d_q)$. Replacing $Q_i$ by $Q_i^{d/d_i}$ if necessary, without loss of generality we may assume that all $Q_i$ are of the same degree $d$. We may also assume that $q>(N-n+1)(n+1)$. 

For a given $v\in S$ and for a fixed point ${\bf x}\in V(k)$, there exists a permutation $(l_{1,v,{\bf x}},...,i_{q,v,{\bf x}})$ of $\{1,2,\ldots,q\}$ such that
$$||Q_{l_{1,v,{\bf x}}}({\bf x})||_v\le ||Q_{l_{2,v,{\bf x}}}({\bf x})||_v\le\cdots\le ||Q_{l_{q,v,{\bf x}}}({\bf x})||_v.$$
We denote by $P_{1,v,{\bf x}},...,P_{n+1,v,{\bf x}}$ the homogeneous polynomials obtained in Lemma \ref{3.1} with respect to $Q_{l_{1,v,{\bf x}}},...,Q_{l_{n+1,v,{\bf x}}}$. It is easy to see that there exists a positive constant $c_1\ge 1$, which is chosen common for all $v\in S,\ {\bf x}\in V(k)$ and $1\le i\le n+1$, such that
$$ ||P_{i,v,{\bf x}}({\bf x})||_v\le c_1\max_{1\le j\le N-n+t}||Q_{l_{j,v,{\bf x}}}({\bf x})||_v, $$
for all $1\le t\le n+1$ and for all ${\bf x}=(x_0:\cdots :x_n)\in V(k)$. 

Since $Q_{1},\ldots,Q_{q}$ are in $N-$subgeneral position w.r.t. $V$, from the compactness of $V$ and the finiteness of $S$ there exists a positive constant $A$ such that
$$ ||{\bf x}||_v^d\le c_2\max_{1\le j\le N+1}||Q_{l_{j,v,{\bf x}}}({\bf x})||_v$$
for all $v\in S$ and ${\bf x}\in V(k)$. Therefore, we have
\begin{align*}
\prod_{i=1}^q\dfrac{||{\bf x}||_v^d}{||Q_i({\bf x})||_v}&\le c_2^{q-N-1}\prod_{j=1}^{N+1}\dfrac{||{\bf x}||_v^d}{||Q_{l_{j,v,{\bf x}}}({\bf x})||_v}\\
&\le c_2^{q-N-1}c_1^{k}\dfrac{||{\bf x}||_v^{(N+1)d}}{\bigl (\prod_{j=1}^{N+1-k}||Q_{l_{j,v,{\bf x}}}({\bf x})||_v\bigl )\cdot\prod_{j=2}^{n+1}||P_{j,v,{\bf x}}({\bf x})||_v}\\
&\le c_2^{q-N-1}c_1^{k}\dfrac{||{\bf x}||_v^{(N+1)d}}{||P_{1,v,{\bf x}}({\bf x})||_v^{N-n+1}\cdot\prod_{j=2}^{n+1}||P_{j,v,{\bf x}}({\bf x})||_v}
\end{align*}
\begin{align*}
&\le c_2^{q-N-1}c_1^{k}c_3^{(N-n)}\dfrac{||{\bf x}||_v^{(N+1)d+(N-n)kd}}{\prod_{j=1}^{n+1}||P_{j,v,{\bf x}}({\bf x})||_v^{N-n+1}},
\end{align*}
where $c_3$ is a positive constant, which is chosen common for all $v\in S$, such that 
$$||P_{j,v,{\bf x}}({\bf x})||_v\le c_3||{\bf x}||_v^d, \ \forall {\bf x}\in V(k).$$
The above inequality implies that
\begin{align}\label{3.2}
\log \prod_{i=1}^q\dfrac{||{\bf x}||_v^d}{||Q_i({\bf x})||_v}\le (N-n+1)\log \dfrac{||{\bf x}||_v^{(n+1)d}}{\prod_{j=1}^{n+1}||P_{j,v,{\bf x}}({\bf x})||_v}+c_4,\ \forall {\bf x}\in V(k),
\end{align}
where $c_4$ is a constant depending only on $Q_i\ (1\le i\le q)$.

We denote by $\mathcal I$ the set of all permutations of $\{1,....,q\}$. Denote by $n_0$ the cardinality of $\mathcal I$. Then we have $n_0=q!$, and we may write that
$\mathcal I=\{I_1,....,I_{n_0}\}$
where $I_i=(l_{i,1},...l_{i,q})\in\mathbb N^q$ and $I_1<I_2<\cdots <I_{n_0}$ in the lexicographic order.
For each $I_i\in\mathcal I$, we denote by $P_{i,1},...,P_{i,{n+1}}$ the homogeneous polynomials obtained in Lemma \ref{3.1} with respect to  $Q_{l_{i,1}},...,Q_{l_{i,N+1}}$.

We consider the mapping $\Phi$ from $V$ into $\P^{l-1}(k)\ (l=n_0(n+1))$, which maps a point ${\bf x}\in V$ into the point $\Phi({\bf x})\in\P^{l-1}(k)$ given by
$$\Phi({\bf x})=(P_{1,1}({\bf x}):\cdots : P_{1,n+1}({\bf x}):P_{2,1}({\bf x}):\cdots:P_{2,n+1}({\bf x}):\cdots:P_{n_0,1}({\bf x}):\cdots :P_{n_0,n+1}({\bf x})).$$ 
Let $Y=\Phi (V)$. Since $V(\overline{k})\cap\bigcap_{j=1}^{n+1}\{P_{1,j}=0\}=\varnothing$, $\Phi$ is a finite morphism on $V$ and $Y$ is a subvariety of $\P^{l-1}(k)$ with $\dim Y=n$ and $\Delta:=\deg Y=\le d^n.\deg V$. 
For every 
$${\bf a} = (a_{1,1},\ldots ,a_{1,n+1},a_{2,1}\ldots,a_{2,n+1},\ldots,a_{n_0,1},\ldots,a_{n_0,n+1})\in\mathbb Z^l_{\ge 0}$$ 
and
$${\bf y} = (y_{1,1},\ldots ,y_{1,n+1},y_{2,1}\ldots,y_{2,n+1},\ldots,y_{n_0,1},\ldots,y_{n_0,n+1})$$ 
we denote ${\bf y}^{\bf a} = y_{1,1}^{a_{1,1}}\ldots y_{1,n+1}^{a_{1,n+1}}\ldots y_{n_0,1}^{a_{n_0,1}}\ldots y_{n_0,n+1}^{a_{n_0,n+1}}$. Let $u$ be a positive integer. We set
\begin{align}\label{3.3}
n_u:=H_Y(u)-1,\ l_u:=\binom{l+u-1}{u}-1,
\end{align}
and define the space
$$ Y_u=k[y_1,\ldots,y_l]_u/(I_Y)_u, $$
which is a vector space of dimension $n_u+1$. We fix a basis $\{v_0,\ldots, v_{n_u}\}$ of $Y_u$ and consider the meromorphic mapping $F$ from $V$ into $\P^{n_u}(k)$ defined by
$$ F({\bf x})=(v_0(\Phi({\bf x})):\cdots :v_{n_u}(\Phi({\bf x}))),\ \ {\bf x}\in V.$$

Fix ${\bf x}\in V$. For a given $v\in S$, without loss of generality we may assume that
$$ ||Q_{l_{1,1}}({\bf x})||_v\le ||Q_{l_{1,2}}({\bf x})||_v\le\cdots\le ||Q_{l_{1,q}}({\bf x})||_v. $$
We define 
$${\bf c}_{v,{\bf x}} = (c_{1,1},\ldots,c_{1,n+1},c_{2,1},\ldots,c_{2,n+1},\ldots,c_{n_0,1},\ldots,c_{n_0,n+1})\in\mathbb R^{l},$$ 
where
\begin{align}\label{3.4}
c_{i,j}:=\log\frac{||{\bf x}||_v^d\cdot||P_{i,j}||_v}{||P_{i,j}({\bf x})||_v}\text{ for } i=1,\ldots ,n_0 \text{ and }j=1,\ldots ,n+1.
\end{align}
We see that $c_{i,j}\ge 0$ for all $v,{\bf x}$ and $j$. By the definition of the Hilbert weight, there are ${\bf a}_{1},\ldots ,{\bf a}_{H_Y(u)}\in\mathbb N^{l}$ with
$$ {\bf a}_{i}=(a_{i,1,1},\ldots,a_{i,1,n+1},\ldots,a_{i,n_0,1},\ldots,a_{i,n_0,n+1}), a_{i,j,s}\in\{1,...,l_u\}, $$
 such that the residue classes modulo $(I_Y)_u$ of ${\bf y}^{{\bf a}_{1}},\ldots ,{\bf y}^{{\bf a}_{H_Y(u)}}$ form a basic of $Y_u$ and
\begin{align}\label{3.5}
S_Y(u,{\bf c}_{v,{\bf x}})=\sum_{i=1}^{H_Y(u)}{\bf a}_{i}\cdot{\bf c}_{v,{\bf x}}.
\end{align}
We see that ${\bf y}^{{\bf a}_{i}}\in Y_m$ (modulo $(I_Y)_m$). Then we may write
$$ {\bf y}^{{\bf a}_{i}}=L_{i,v,{\bf x}}(v_0,\ldots ,v_{H_Y(u)}), $$ 
where $L_{i,v,{\bf x}}\ (1\le i\le H_Y(u))$ are independent linear forms.
We have
\begin{align*}
\log\prod_{i=1}^{H_Y(u)} ||L_{i,v,{\bf x}}(F({\bf x}))||_v&=\log\prod_{i=1}^{H_Y(u)}\prod_{\overset{1\le i\le n_0}{1\le j\le n+1}}||P_{i,j}({\bf x})||_v^{a_{i,j}}\\
&=-S_Y(m,{\bf c}_{v,{\bf x}})+duH_Y(u)\log ||{\bf x}||_v +O(uH_Y(u)).
\end{align*}
This implies that
\begin{align*}
\log\prod_{i=1}^{H_Y(u)}\dfrac{||F({\bf x})||_v\cdot ||L_{i,v,{\bf x}}||_v}{||L_{i,v,{\bf x}}(F({\bf x}))||_v}=&S_Y(u,{\bf c}_{v,{\bf x}})-duH_Y(u)\log ||{\bf x}||_v \\
&+H_Y(u)\log ||F({\bf x})||_v+O(uH_Y(u)).
\end{align*}
We note that $L_{i,v,{\bf x}}$ depends on $i$ and ${\bf x}$, but the number of these linear forms is finite. We denote by $\mathcal L_v$ the set of all $L_{i,v,{\bf x}}$ occuring in the above inequalities (when $i$ and ${\bf x}$ vary). Then we have
\begin{align}\label{3.6}
\begin{split}
S_Y(u,{\bf c}_{v,{\bf x}})\le&\max_{\mathcal J\subset\mathcal L_v}\log\prod_{L\in \mathcal J}\dfrac{||F({\bf x})||_v\cdot ||L||_v}{||L(F({\bf x}))||_v}+duH_Y(u)\log ||{\bf x}||_v\\
& -H_Y(u)\log ||F({\bf x})||_v+O(uH_Y(u)),
\end{split}
\end{align}
where the maximum is taken over all subsets $\mathcal J\subset\mathcal L_v$ with $\sharp\mathcal J=H_Y(u)$ and $\{L;L\in\mathcal J\}$ is linearly independent.
From Theorem \ref{2.12} we have
\begin{align}\label{3.7}
\dfrac{1}{uH_Y(u)}S_Y(u,{\bf c}_{v,{\bf x}})\ge&\frac{1}{(n+1)\Delta}e_Y({\bf c}_{v,{\bf x}})-\frac{(2n+1)\Delta}{u}\max_{\underset{1\le i\le n_0}{1\le j\le n+1}}c_{i,j}
\end{align}
It is clear that
\begin{align*}
\max_{\underset{1\le i\le n_0}{1\le j\le n+1}}c_{i,j}\le \sum_{1\le j\le n+1}\log\frac{||{\bf x}||_v^d\cdot||P_{1,j}||_v}{||P_{1,j}({\bf x})||_v}+c_4,
\end{align*}
where $c_4$ is a constant not depending on ${\bf x}$ and $v$. Combining (\ref{3.6}), (\ref{3.7}) and the above remark, we get
\begin{align}\nonumber
\frac{1}{(n+1)\Delta}e_Y({\bf c}_{v,{\bf x}})\le &\dfrac{1}{uH_Y(u)}\left (\max_{\mathcal J\subset\mathcal L_v}\log\prod_{L\in \mathcal J}\dfrac{||F({\bf x})||_v\cdot ||L||_v}{||L(F({\bf x}))||_v}-H_Y(u)\log ||F({\bf x})||_v\right )\\
\label{3.8}
\begin{split}
&+d\log ||{\bf x}||_v+\frac{(2m+1)\Delta}{n}\max_{\underset{1\le i\le n_0}{1\le j\le n+1}}c_{i,j}+O(1/n)\\
\le &\dfrac{1}{uH_Y(u)}\left (\max_{\mathcal J\subset\mathcal L_v}\prod_{L\in\mathcal J}\dfrac{||F({\bf x})||_v\cdot ||L||_v}{||L(F({\bf x}))||_v}-H_Y(u)\log ||F({\bf x})||_v\right )\\
&+d\log ||{\bf x}||_v+\frac{(2n+1)\Delta}{m}\sum_{1\le j\le k}\log\frac{||{\bf x}||_v^d\cdot||P_{1,j}||_v}{||P_{1,j}({\bf x})||_v}+O(1/n).
\end{split}
\end{align}
Since $P_{1,1},\ldots,P_{1,n+1}$ are in general with respect to $X$, By Lemma \ref{2.13}, we have
\begin{align}\label{3.9}
e_Y({\bf c}_{v,{\bf x}})\ge (c_{1,1}+\cdots +c_{1,n+1})\cdot\Delta =\left (\sum_{1\le j\le k}\log\frac{||{\bf x}||_v^d\cdot||P_{1,j}||_v}{||P_{1,j}({\bf x})||_v}\right )\cdot\Delta.
\end{align}
Then, from (\ref{3.2}), (\ref{3.8}) and (\ref{3.9}) we have
\begin{align}\label{3.10}
\begin{split}
\frac{1}{N-n+1}&\left (\log \prod_{i=1}^q\dfrac{||{\bf x}||_v^d}{||Q_i({\bf x})||_v}+O(1)\right )\\
&\le\dfrac{n+1}{uH_Y(u)}\left (\max_{\mathcal J\subset\mathcal L_v}\log\prod_{L\in\mathcal J}\dfrac{||F({\bf x})||_v\cdot ||L||_v}{||L(F({\bf x}))||_v}-H_Y(u)\log ||F({\bf x})||_v\right )\\
&+d(n+1)\log ||{\bf x}||_v+\frac{(2n+1)(n+1)\Delta}{u}\sum_{\underset{1\le i\le n_0}{1\le j\le n+1}}\log\frac{||{\bf x}||_v^d\cdot||P_{i,j}||_v}{||P_{i,j}({\bf x})||_v},
\end{split}
\end{align}
where the term $O(1)$ does not depend on ${\bf x}$. 

Summing-up both sides of the above inequalities over all $v\in S$, we obtain
\begin{align}\nonumber
\frac{1}{N-n+1}\sum_{v\in S}\sum_{j=1}^q\lambda_{Q_j,v}(x)\le &\dfrac{n+1}{uH_Y(u)}\sum_{v\in S}\max_{\mathcal J\subset\mathcal L_v}\log\prod_{L\in\mathcal J}\dfrac{||F({\bf x})||_v\cdot ||L||_v}{||L(F({\bf x}))||_v}\\
\label{3.11}
\begin{split}
&-\dfrac{n+1}{u}\sum_{v\in S}\log ||F({\bf x})||_v+d(n+1)\sum_{v\in S}\log ||{\bf x}||_v\\
&+\frac{(2n+1)(n+1)\Delta}{u}\sum_{\underset{1\le i\le n_0}{1\le j\le n+1}}\sum_{v\in S}\log\frac{||{\bf x}||_v^d\cdot||P_{i,j}||_v}{||P_{i,j}({\bf x})||_v}+O(1).
\end{split}
\end{align} 
Since the above inequality does not depend on the choice of components of ${\bf x}$, we may assume that all components of ${\bf x}$ are $S$-integer points and all coefficients of $F$ are $S$-integer points. Then we have
$$ \sum_{v\in S}\log ||{\bf x}||_v=h({\bf x}) \ \text{ and } \sum_{v\in S}\log ||F({\bf x})||_v=h(F({\bf x})).$$
We also have
$$ \sum_{v\in S}\log\frac{||{\bf x}||_v^d\cdot||P_{i,j}||_v}{||P_{i,j}({\bf x})||_v}\le \sum_{v\in M(k)}\log\frac{||{\bf x}||_v^d\cdot||P_{i,j}||_v}{||P_{i,j}({\bf x})||_v}=dh({\bf x}).$$
On the other hand by the subspace theorem due to Schlickewei \cite{Sch77} (see also Schmidt \cite[Theorem 3]{Sch75}), for every $\epsilon' >0$ and all $F({\bf x})$ outside a finite union of proper linear subspaces, we have
\begin{align*}
\sum_{v\in S}\max_{\mathcal J\subset\mathcal L_v}\log\prod_{L\in\mathcal J}\dfrac{||F({\bf x})||_v\cdot ||L||_v}{||L(F({\bf x}))||_v}\le (H_Y(u)+\epsilon') h(F({\bf x})).
\end{align*}
Combining this inequality with (\ref{3.11}), we have
\begin{align}\label{3.12}
\begin{split}
\frac{1}{N-n+1}\sum_{v\in S}\sum_{j=1}^q\lambda_{Q_j,v}({\bf x})&\le \dfrac{\epsilon'(n+1)}{uH_Y(u)}h(F({\bf x}))\\
&+\left(d(n+1)+\frac{(2n+1)(n+1)l\Delta}{u}\right)h({\bf x})+O(1),
\end{split}
\end{align}
for all ${\bf x}\in V$ outside a union of finite proper algebraic subsets.

We note that $h(F({\bf x}))\le uh({\bf x})$. Then taking
$$ u\ge 4(N-n+1)(2n+1)(n+1)l\Delta \epsilon^{-1} \ \text{ and } \epsilon'\le \frac{dH_Y(u)\epsilon}{4(N-n+1)(n+1)},$$
from (\ref{3.12}) we get
$$ \sum_{v\in S}\sum_{j=1}^q\dfrac{\lambda_{Q_j,v}({\bf x})}{d}\le  \left((N-n+1)(n+1)+\frac{\epsilon}{2}\right)h({\bf x})+O(1),$$
for all ${\bf x}\in V$ outside a union of finite proper algebraic subsets.
Since the height function $h({\bf x})$ is unbounded and there only finite points ${\bf x}\in V$ such that $h({\bf x})$ is bounded above by a certain positive constant, if we replace $\frac{\epsilon}{2}$ by $\epsilon$ then the term $O(1)$ in the above inequality can be absorbed. The theorem is proved.
\end{proof}

%\noindent
%{\bf Acknowledgements.} This research is funded by Vietnam National Foundation for Science and Technology Development (NAFOSTED) under grant number 101.04-2015.03.

\vskip0.2cm
{\footnotesize 
\noindent
{\sc Si Duc Quang}\\
Department of Mathematics,\\
Hanoi National University of Education,\\
136-Xuan Thuy, Cau Giay, Hanoi, Vietnam.\\
\textit{E-mail}: quangsd@hnue.edu.vn

\end{document}